\pdfoutput=1
\RequirePackage{ifpdf}
\ifpdf 
\documentclass[pdftex]{sigma}
\else
\documentclass{sigma}
\fi

\numberwithin{equation}{section}

\newtheorem{Theorem}{Theorem}[section]

\newtheorem{Proposition}[Theorem]{Proposition}
{ \theoremstyle{definition}

\newtheorem{Remark}[Theorem]{Remark} }

\begin{document}


\renewcommand{\thefootnote}{$\star$}

\newcommand{\arXivNumber}{1510.05770}

\renewcommand{\PaperNumber}{035}

\FirstPageHeading

\ShortArticleName{Generalized Stieltjes Transforms of Compactly-Supported Probability Distributions}

\ArticleName{Generalized Stieltjes Transforms\\ of Compactly-Supported Probability Distributions:\\ Further Examples\footnote{This paper is a~contribution to the Special Issue on Orthogonal Polynomials, Special Functions and Applications.
The full collection is available at \href{http://www.emis.de/journals/SIGMA/OPSFA2015.html}{http://www.emis.de/journals/SIGMA/OPSFA2015.html}}}

\Author{Nizar DEMNI}

\AuthorNameForHeading{N.~Demni}

\Address{IRMAR, Universit\'e de Rennes 1, Campus de Beaulieu, 35042 Rennes cedex, France}
\Email{\href{mailto:nizar.demni@univ-rennes1.fr}{nizar.demni@univ-rennes1.fr}}
\URLaddress{\url{https://perso.univ-rennes1.fr/nizar.demni}}

\ArticleDates{Received December 12, 2015, in f\/inal form April 06, 2016; Published online April 12, 2016}

\Abstract{For two families of beta distributions, we show that the generalized Stieltjes transforms of their elements may be written as elementary functions (powers and fractions) of the Stieltjes transform of the Wigner distribution. In particular, we retrieve the examples given by the author in a previous paper and relating generalized Stieltjes transforms of special beta distributions to powers of (ordinary) Stieltjes ones. We also provide further examples of similar relations which are motivated by the representation theory of symmetric groups. Remarkably, the power of the Stieltjes transform of the symmetric Bernoulli distribution is a generalized Stietljes transform of a probability distribution if and only if the power is greater than one. As to the free Poisson distribution, it corresponds to the product of two independent Beta distributions in $[0,1]$ while another example of Beta distributions in $[-1,1]$ is found and is related with the Shrinkage process. We close the exposition by considering the generalized Stieltjes transform of a linear functional related with Humbert polynomials and generalizing the symmetric Beta distribution. }


\Classification{33C05; 33C20; 33C45; 44A15; 44A20}

\renewcommand{\thefootnote}{\arabic{footnote}}
\setcounter{footnote}{0}

\section{Reminder}
Let $\lambda$ be a positive real number and $\mu$ be a probability measure supported in the real line (possibly depending on $\lambda$). Its generalized Stieltjes transform is def\/ined by
\begin{gather*}
G_{\lambda, \mu}(z) := \int \frac{1}{(z-x)^{\lambda}}\mu(dx)
\end{gather*}
for complex numbers $z$ lying outside the support of $\mu$ and such that $(z-x)^{\lambda}$ is the principal branch of the power function. In particular, $G_{1,\mu}$ is the (ordinary) Stieltjes transform of $\mu$ which we simply denote hereafter by $G_{\mu}$. The latter has deep connections with the theory of continued fractions, combinatorics, Pad\'e approximations and free probability theory (see~\cite{Ism, NS} and references therein). As to $G_{\lambda, \mu}$, much less is known. For instance, this transform was introduced in \cite{Wid} for measures $\mu$ supported in the positive half-line and an inversion formula was derived. For the same type of measures, the notion of exact order is def\/ined in~\cite{Kar-Pri} and~\cite{Kar-Pri1} and determined for hypergeometric series. When $\lambda$ is a positive integer, the random weighted average of independent random variables gives rise to generalized Stieltjes transforms of continuous probability measures which may be written as products of Stieltjes transforms \cite{Hom-Sol,Roo-Sol, Van}). In the recent paper~\cite{Koo}, generalized Stieltjes transforms appear as transmutation operators between the solutions of the hypergeometric dif\/ferential equation. Motivated by a~generalization of free probability theory and relying on the characterization of ultraspherical-type generating series for orthogonal polynomials \cite{AlS-Ver,Demni0}, we provided in~\cite{Demni1} four examples of compactly-supported probability measures $\mu = \mu_{\lambda}$ for which there exist probability measures $\nu = \nu_{\lambda}$ satisfying
\begin{gather}\label{Rel}
G_{\lambda, \mu}(z) = [G_{1,\nu}(z)]^{\lambda} = [G_{\nu}(z)]^{\lambda}.
\end{gather}
Actually, the Wigner distribution plays a central role in (Voiculescu) free probability theory and is an instance of the symmetric beta distribution. Since any symmetric probability distribution with f\/inite moments of all orders arises as the weak limit of sums of self-adjoint variables in a~suitable algebra \cite{Cab-Ion}, it is then natural to seek a $\lambda$-deformation of free probability theory where the symmetric beta distribution appears in the central limit theorem. At the analytic side, when relations like~\eqref{Rel} hold, they bridge between the free additive convolution and its $\lambda$-deformation provided that the measure~$\nu$ is independent of~$\lambda$. As shown in~\cite{Demni1}, this last property is satisf\/ied by the couple of symmetric beta and Wigner distributions and this is the only example derived there with this property. This elementary observation raises the following questions:
\begin{itemize}\itemsep=0pt
\item Given $\lambda > 0$ and $\mu = \mu_{\lambda}$, what are the necessary and/or suf\/f\/icient conditions ensuring~\eqref{Rel} to hold with $\nu$ being independent of $\lambda$?
\item Conversely, given $\lambda > 0$ and $\nu$ independent of $\lambda$, when does the power $[G_{\nu}]^{\lambda}$ is a generalized Stieltjes transform of a probability distribution $\mu = \mu_{\lambda}$?
\item Is the symmetric beta distribution the only element in the family of beta distributions satisfying~\eqref{Rel} with $\nu$ being independent of $\lambda$?
\item In the same vein, f\/ind other or all examples of beta distributions satisfying~\eqref{Rel} with this property?
\end{itemize}
Most likely, it is much easier to look for partial or def\/initive answers to the two last questions rather than the remaining ones. In general, if we want to derive an expression of~$G_{\lambda,\mu}$ for given~$\lambda$ and~$\mu$, it suf\/f\/ices to expand
\begin{gather*}
(z,x) \mapsto \frac{1}{(z-x)^{\lambda}},
\end{gather*}
in a absolutely convergent series of orthogonal polynomials with respect to $\mu$ since then $G_{\lambda,\mu}$ is nothing else but the f\/irst term of this series. Recently, this task was achieved in \cite{Cohl1} for the family of Jacobi polynomials whence the generalized Stieltjes transform of any beta distribution in $[-1,1]$ follows after a simple integration. Using linear and quadratic transformations of the Gauss hypergeometric function, we retrieve the few examples of \eqref{Rel} given in \cite{Demni1} and prove for two larger classes of beta distributions that their generalized Stieltjes transforms are elementary functions (powers and fractions) of the Stieltjes transform of the Wigner distribution. We also provide a new example of beta distributions satisfying~\eqref{Rel}, where~$\nu$ does not depend on~$\lambda$ as well thereby giving a negative answer to the third question. Actually, the probability distribution~$\nu$ describes the large time behavior of the Shrinkage process and is the transition measure of triangular diagrams~\cite{Ker}. Two other examples of probability measures~$\nu$ occurring in representation theory of symmetric groups are considered. In the former, $\nu$~is the symmetric Bernoulli distribution which is the transition measure of a square Young diagram~\cite{Biane} and~\eqref{Rel} holds if and only if~$\lambda \geq 1$. The latter shows that the $\lambda$-power of the Stieltjes transform of the free Poisson distribution (which arises in the decomposition of the tensor product representation of the symmetric group~\cite{Biane1}) is the generalized Stieltjes transform of a product of two independent beta distributions supported in~$[0,1]$. Indeed, the multiplicative convolution with a beta distribution of special parameters provides a transformation preserving probability distributions satisfying~\eqref{Rel} for any $\lambda > 0$. In the last part of the paper, we investigate a~possible generalization of the relation~\eqref{Rel} holding between the symmetric beta and the Wigner distributions. More precisely, we consider the generating series of Humbert polynomials which are $d$-orthogonal with respect to $d \geq 1$ linear functionals and integrate it with respect to the f\/irst functional. Doing so leads to the problem of f\/inding a suitable root of a trinomial equation of degree $d+1$ which is known to be expressed by means of the Gauss hypergeometric function.

\section{Special functions}
In this paragraph, we recall the def\/initions and some properties of special functions occurring in the remainder of the paper. We refer the reader to \cite{Erd, Ism,Rai,Man-Sri}. Let $\Gamma$ denote the gamma function and recall the Legendre duplication formula:
\begin{gather*}
\sqrt{\pi} \Gamma(2z) = 2^{2z-1}\Gamma(z)\Gamma\left(z+\frac{1}{2}\right).
\end{gather*}
For a complex number $z$ and a positive integer $k \geq 1$, the Pochhammer symbol is def\/ined by
\begin{gather*}
(z)_k = z(z+1)\cdots(z+k-1)
\end{gather*}
with the convention $(z)_0 = 1$. When $z$ is not a negative integer, we can write
\begin{gather*}
(z)_k = \frac{\Gamma(z+k)}{\Gamma(z)},
\qquad \text{while}
\qquad
(-n)_k = \frac{(-1)^kn!}{(n-k)!}
\end{gather*}
for any positive integer $n \geq k$ and vanishes otherwise. Next, the hypergeometric series ${}_pF_q$ is def\/ined by
\begin{gather*}
{}_pF_q(a_1, \dots, a_p, b_1, \dots, b_q; z) = \sum_{n \geq 0} \frac{(a_1)_n\cdots (a_p)_n}{(b_1)_n\cdots (b_q)_n} \frac{z^n}{n!}
\end{gather*}
when the series converges. In particular, the Gauss hypergeometric series ${}_2F_1$ converges in the open unit disc $\{|z| < 1\}$ and has an analytic extension to the complex plane cut along the half line $[1,\infty)$. This function will play a major role in our computations and we use further the following linear and quadratic transformations valid for $|\arg(1-u)| < \pi$:
\begin{gather}\label{Lin1}
{}_2F_1(a,c+d,c; u) = \frac{1}{(1-u)^{a}}\, {}_2F_1\left(a,-d,c; \frac{u}{u-1}\right),
\\
\label{Quad0}
{}_2F_1 (a,b,2a; u) = \frac{1}{(1-(u/2))^b}\, {}_2F_1\left(\frac{b}{2}, \frac{b+1}{2}, a+\frac{1}{2}; \frac{u^2}{(2-u)^2}\right),
\\
\label{Quad1}
{}_2F_1\left(a,a+\frac{1}{2}, b; u\right) = \frac{2^{2a}}{(1+\sqrt{1-u})^{2a}}\, {}_2F_1\left(2a, 2a-b+1, b; \frac{u}{(1+\sqrt{1-u})^{2}}\right).
\end{gather}
In particular, \eqref{Quad1} yields the closed formulas
\begin{gather}\label{Id1}
{}_2F_1\left(a-\frac{1}{2}, a, 2a; u\right) = \frac{2^{2a-1}}{(1+\sqrt{1-u})^{2a-1}},
\end{gather}
and
\begin{gather}\label{Id2}
{}_2F_1\left(a, a+ \frac{1}{2}, 2a; u\right) = \frac{1}{\sqrt{1-u}}\frac{2^{2a-1}}{(1+\sqrt{1-u})^{2a-1}}.
\end{gather}
We shall also make use of the Euler integral representation of the ${}_2F_1$: for $\Re(c) > \Re(b) > 0$ and $|z| < 1$,
\begin{gather}\label{Euler}
{}_2F_1(a,b,c; z) = \frac{\Gamma(c)}{\Gamma(c-b)\Gamma(b)} \int_0^1(1-uz)^{-a}u^{b-1}(1-u)^{c-b-1} du.
\end{gather}
As to hypergeometric polynomials, we denote by $P_n^{(\gamma, \beta)}$, $C_n^{(\alpha)}$ the $n$-th Jacobi and ultraspherical respectively:
\begin{gather*}
P_n^{(\gamma, \beta)}(x) = \frac{(\gamma+1)_n}{n!}\, {}_2F_1\left(-n, n+\gamma+\beta+1, \gamma+1; \frac{1-x}{2}\right), \qquad \gamma, \beta > -1, \\
C_n^{(\alpha)}(x) = \frac{(2\alpha)_n}{(\alpha+1/2)_n} P_n^{(\alpha-1/2, \alpha-1/2)}(x), \qquad \alpha > -1/2, \qquad \alpha \neq 0.
\end{gather*}

\section{Generalized Stieltjes transform of beta distributions}

The starting point of our investigations is the following expansion proved in \cite[Theorem~1]{Cohl1}: for any $z \in \mathbb{C} {\setminus} (-\infty, 1]$ on any given ellipse with foci at $\pm 1$ and any $x$ in the interior of this ellipse,
\begin{gather*}
\frac{1}{(z-x)^{\lambda}} = \sum_{n = 0}^{\infty}\frac{\Gamma(\gamma+\beta+n+1)(\lambda)_n}{\Gamma(2n+\gamma+\beta+1)} \frac{2^n}{(z-1)^{n+\lambda}}\\
\hphantom{\frac{1}{(z-x)^{\lambda}} =}{}\times
 {}_2F_1\left(n+\lambda, n+\gamma+1, 2n+2+\gamma+\beta, \frac{2}{1-z}\right)P_n^{(\gamma, \beta)}(x).
\end{gather*}
For f\/ixed $z$, this series converges uniformly in $x \in [-1,1]$ (see, e.g., \cite[Chapter~IX, Theorem~9.1.1]{Sze}) and as shown in~\cite{Cohl1}, it is valid for a set of parameters $\lambda$, $\gamma$, $\beta$ containing $\mathbb{R}_+^* \times (-1,\infty)^2$. However, we shall restrict ourselves to the latter which is suf\/f\/icient for our purposes and subsequent computations. As a matter of fact, the orthogonality of Jacobi polynomials with respect to the beta distribution
\begin{gather*}
\mu_{\gamma,\beta}(dx) := \frac{\Gamma(\gamma+\beta+2)}{2^{\gamma+\beta+1}\Gamma(\gamma+1)\Gamma(\beta+1)} (1-x)^{\gamma}(1+x)^{\beta}{\bf 1}_{[-1,1]}(x)dx
\end{gather*}
readily gives\footnote{We could also use the Euler integral representation of the ${}_2F_1$ hypergeometric function.}
\begin{gather}\label{GST1}
G_{\lambda, \mu_{\gamma,\beta}}(z) = \frac{1}{(z-1)^{\lambda}}\,{}_2F_1\left(\lambda, \gamma+1, \gamma+\beta+2, \frac{2}{1-z}\right).
\end{gather}
Note that the r.h.s.\ of $\eqref{GST1}$ is not symmetric in $(\gamma, \beta)$ unless $\gamma = \beta$. Nonetheless, we can use~\eqref{Lin1} to transform~\eqref{GST1} into
\begin{gather}\label{GST2}
G_{\lambda, \mu_{\gamma,\beta}}(z) = \frac{1}{(z+1)^{\lambda}}\, {}_2F_1\left(\lambda, \beta+1, \gamma+\beta+2, \frac{2}{1+z}\right).
\end{gather}
Here, one performs the transformation for $z$ lying in a suitable region in the complex plane (in order to use the principal determinations of $(z+1)^{\lambda}$, $(z-1)^{\lambda}$, $[(z+1)/(z-1)]^{\lambda}$) and then extends the obtained equality analytically to $\mathbb{C} {\setminus} (-\infty, 1]$. With~\eqref{GST1} and~\eqref{GST2} in hands, we start our investigations of generalized Stieltjes transforms of two families of beta distributions.

\subsection{Symmetric beta distributions}
Take $\gamma=\beta = \lambda-(1/2)$, $\lambda > 0$ in \eqref{GST1}. Then \eqref{Id1} entails
\begin{gather*}
\frac{1}{(z-1)^{\lambda}}\, {}_2F_1\left(\lambda, \lambda+\frac{1}{2}, 2\lambda+1, \frac{2}{1-z}\right) = \left[\frac{2}{z+\sqrt{z^2-1}}\right]^{\lambda}.
\end{gather*}
In the right-hand side of the last expression, we recognize the Stieltjes transform of the Wigner distribution
\begin{gather*}
G_{\rm W}(z) = \int_{-1}^1 \frac{1}{z-x} \frac{2\sqrt{1-x^2}}{\pi}dx = \frac{2}{z+\sqrt{z^2-1}}, \qquad z > 1.
\end{gather*}
As a result,
\begin{gather}\label{Examp1}
\frac{\Gamma(\lambda+1)}{\sqrt{\pi}\Gamma(\lambda+1/2)} \int_{-1}^1\frac{(1-x^2)^{\lambda-1/2}}{(z-x)^{\lambda}}dx = [G_{\rm W}(z)]^{\lambda},
\end{gather}
which is the f\/irst example given in~\cite{Demni1}. In order to retrieve the second example given in~\cite{Demni1}, specialize~\eqref{GST1} to $\gamma = \beta = \lambda - (3/2)$, $\lambda > 1$ and use~\eqref{Id2} to derive
\begin{gather*}
\frac{1}{(z-1)^{\lambda}}\, {}_2F_1\left(\lambda-\frac{1}{2}, \lambda, 2\lambda-1, \frac{2}{1-z}\right) = \frac{1}{\sqrt{z^2-1}}[G_{\rm W}(z)]^{\lambda-1}.
\end{gather*}
But the Stieltjes transform of the arcsine distribution reads
\begin{gather*}
G_{\rm AS}(z) = \int_{-1}^1 \frac{1}{z-x} \frac{dx}{\pi \sqrt{1-x^2}} = \frac{1}{\sqrt{z^2-1}}.
\end{gather*}
Consequently, for any $\lambda > 1$,
\begin{gather}\label{Examp2}
\frac{\Gamma(\lambda)}{\sqrt{\pi}\Gamma(\lambda-1/2)}\int_{-1}^1 \frac{(1-x^2)^{\lambda-3/2}}{(z-x)^{\lambda}} dx = G_{\rm AS}(z)[G_{\rm W}(z)]^{\lambda-1}.
\end{gather}
These two examples are indeed instances of the following more general formula:
\begin{Proposition}
If $\lambda = \gamma + 1/2 + k \geq 0$ for some integer $k \geq 0$, then the generalized Stieltjes transform of
\begin{gather*}
\mu_{\gamma, \gamma}(dx) = \frac{\Gamma(\gamma+3/2)}{\sqrt{\pi}\Gamma(\gamma+1)} \big(1-x^2\big)^{\gamma}{\bf 1}_{[-1,1]}(x)dx
\end{gather*}
may be written by means of powers and fractions of the variable $G_{\rm W}$:
\begin{gather*}
G_{\lambda, \mu_{\gamma, \gamma}}(z) = \frac{4^k[G_{\rm W}(z)]^{\lambda}}{\left[4-[G_{\rm W}(z)]^2\right]^{k}}\, {}_2F_1\left(k, 1-k, \gamma + \frac{3}{2}; \frac{[G_{\rm W}(z)]^2}{[G_{\rm W}(z)]^2-4}\right).
\end{gather*}
\end{Proposition}

\begin{proof}
Specialize \eqref{GST1} to $\gamma = \beta$ and use \eqref{Quad0} to get
\begin{gather*}
\frac{1}{(z-1)^{\lambda}}\,{}_2F_1\left(\lambda, \gamma+1, 2\gamma+2, \frac{2}{1-z}\right) = \frac{1}{z^{\lambda}}\,{}_2F_1\left(\frac{\lambda}{2}, \frac{\lambda+1}{2}, \gamma+\frac{3}{2}, \frac{1}{z^2}\right).
\end{gather*}
Assuming $\gamma +(1/2) \leq \lambda$ and using \eqref{Quad1} transforms the r.h.s.\ of the last equality to
\begin{gather*}
\frac{1}{z^{\lambda}}\, {}_2F_1\left(\frac{\lambda}{2}, \frac{\lambda+1}{2}, \gamma+\frac{3}{2}, \frac{1}{z^2}\right) = [G_{\rm W}(z)]^{\lambda} \, {}_2F_1\left(\lambda, \lambda-\gamma - \frac{1}{2}, \gamma + \frac{3}{2}; \frac{[G_{\rm W}(z)]^2}{4}\right).
\end{gather*}
Assuming further that $\lambda = \gamma + 1/2 + k$ for some integer $k \geq 0$ and noting that
\begin{gather*}
1-\frac{[G_{\rm W}(z)]^2}{4} = 2\frac{\sqrt{z^2-1}}{z+\sqrt{z^2-1}} \in [0,1), \qquad z > 1,
\end{gather*}
then we can appeal to the linear transformation \eqref{Lin1} and end up with
\begin{gather*}
{}_2F_1\left(k, \lambda, \gamma + \frac{3}{2}; \frac{[G_{\rm W}(z)]^2}{4}\right) = \frac{4^k}{\left[4-[G_{\rm W}(z)]^2\right]^{k}}\, {}_2F_1\left(k, 1-k, \gamma + \frac{3}{2}; \frac{[G_{\rm W}(z)]^2}{[G_{\rm W}(z)]^2-4}\right).\tag*{\qed}
\end{gather*}
\renewcommand{\qed}{}
\end{proof}

\begin{Remark}
Note that
\begin{gather*}
1-\frac{[G_{\rm W}(z)]^2}{4} = \sqrt{z^2-1}G_{\rm W}(z) = \frac{G_{\rm W}(z)}{G_{\rm AS}(z)}
\end{gather*}
so that
\begin{gather*}
G_{\lambda, \mu_{\lambda -(1/2) - k, \lambda- (1/2)-k}}(z) = [G_{\rm W}(z)]^{\lambda-1}[G_{\rm AS}(z)]^{k}\, {}_2F_1\left(k, 1-k, \gamma + \frac{3}{2}; \frac{G_{\rm AS}(z)G_{\rm W}(z)}{4}\right).
\end{gather*}
\end{Remark}

In the next paragraph, we derive a similar formula for a family of nonsymmetric beta distributions.

\subsection{Nonsymmetric beta distributions}
In \cite{Demni1}, two other examples of probability distributions satisfying~\eqref{Rel} were derived. They correspond to nonsymmetric beta distributions with parameters
\begin{gather}\label{Parameters}
(\gamma, \beta) = \left(\lambda - \frac{1}{2}, \lambda - \frac{3}{2}\right), \qquad (\gamma, \beta) = \left(\lambda - \frac{3}{2}, \lambda - \frac{1}{2}\right)
\end{gather}
and readily follow from \eqref{GST1} together with the identities~\eqref{Id1} and~\eqref{Id2}. As with the previous family of symmetric beta distributions, we can derive a more general formula for generalized Stieltjes transforms of nonsymmetric ones with
parameters $\gamma = \lambda -1/2$ and $\beta = \lambda-k-1/2 = \gamma - k$, $k \geq 1$:
\begin{Proposition}
For any $z > 1$,
\begin{align*}
G_{\lambda, \mu_{\lambda-(1/2),\lambda-(1/2)-k}}(z) = \frac{[G_{\rm W}(z)]^{\lambda-(k/2)}}{(1+z)^{k/2}}\, {}_2F_1\left(1-k, k, 2\lambda-k+1; \frac{G_{\rm W}(z)}{G_{\rm W}(z)+2}\right).
\end{align*}
The equality extends analytically to the complex plane cut along $(-\infty, 1]$.
\end{Proposition}

\begin{proof}
Using \eqref{Quad1} and \eqref{Lin1}, we get
\begin{align*}
G_{\lambda, \mu_{\lambda-(1/2),\lambda-(1/2)-k}}(z) &= \frac{1}{(z-1)^{\lambda}}\,{}_2F_1\left(\lambda, \lambda + \frac{1}{2}, \lambda +\beta+\frac{3}{2}, \frac{2}{1-z}\right)
\\ &= [G_{\rm W}(z)]^{\lambda}\, {}_2F_1\left(2\lambda, \lambda-\beta-\frac{1}{2}, \lambda+\beta+\frac{3}{2}; -\frac{G_{\rm W}(z)}{2}\right)
\\& = \frac{2^k[G_{\rm W}(z)]^{\lambda}}{[2+G_{\rm W}(z)]^{k}}\, {}_2F_1\left(1-k, k, 2\lambda-k+1; \frac{G_{\rm W}(z)}{G_{\rm W}(z)+2}\right).
\end{align*}
From the relation
\begin{gather*}
1+\frac{[G_{\rm W}(z)]^2}{4} = zG_{\rm W}(z),
\end{gather*}
it follows that
\begin{gather*}
\left(1+\frac{G_{\rm W}(z)}{2}\right)^2 = (z+1)G_{\rm W}(z)
\end{gather*}
and the proposition follows.
\end{proof}

\begin{Remark}
Interchanging the roles of $\gamma$ and $\beta$ with the help of \eqref{GST2}, similar computations yield
\begin{align*}
G_{\lambda, \mu_{\lambda-(1/2) - k,\lambda-(1/2)}}(z) &= \frac{2^k[G_{\rm W}(z)]^{\lambda}}{[2-G_{\rm W}(z)]^{k}}\, {}_2F_1\left(1-k, k, 2\lambda-k+1; \frac{G_{\rm W}(z)}{G_{\rm W}(z)-2}\right)
\\& = \frac{[G_{\rm W}(z)]^{\lambda - (k/2)}}{(z-1)^{k/2}}\, {}_2F_1\left(1-k, k, 2\lambda-k+1; \frac{G_{\rm W}(z)}{G_{\rm W}(z)-2}\right).
\end{align*}
\end{Remark}

\section[Generalized Stieltjes transforms as powers of Stieltjes transforms: further examples]{Generalized Stieltjes transforms as powers\\ of Stieltjes transforms: further examples}

So far, we dispose of four couples $(\mu,\nu)$ of probability distributions satisfying \eqref{Rel}: those displayed in~\eqref{Examp1} and~\eqref{Examp2} and those corresponding to the two couples of parameters specif\/ied in~\eqref{Parameters}. However, only~\eqref{Examp1} has the property that $\nu$ (the Wigner distribution in this case) does not depend on $\lambda$. In this section, we derive three more examples enjoying this property which, like the Wigner distribution, appeared as transition distributions of Young or continuous diagrams (see~\cite{Ker} for more details).

\subsection{The square diagram}
The symmetric Bernoulli distribution
\begin{gather*}
\nu = \frac{1}{2}[\delta_{-L} + \delta_L], \qquad L \in \mathbb{N} {\setminus} \{0\},
\end{gather*}
is the basic example of transition distribution of a Young diagram. Indeed, it corresponds to the square diagram of width~$L$~\cite{Biane}. To simplify, take $L=1$ so that
 \begin{gather*}
G_{\nu}(z) = \frac{z}{z^2-1}.
\end{gather*}
Then we shall prove
\begin{Proposition}
If $\nu$ is the symmetric Bernoulli distribution, then
\begin{gather*}
G_{\lambda, \mu}(z) = [G_{\nu}(z)]^{\lambda}
\end{gather*}
for every $\lambda \geq 1$, where
\begin{gather*}
\mu(dx) = \mu_{\lambda}(dx) = \frac{1}{2^{\lambda}} \left[[\delta_1+\delta_{-1}](dx) + \frac{1}{\sqrt{|x|}}h(\sqrt{|x|})dx\right]
\end{gather*}
and
\begin{gather*}
h(x) = h_{\lambda}(x) := \frac{\lambda(\lambda-1)}{4}x^{\lambda-1}\, {}_2F_1\left(\frac{\lambda}{2}+1, \frac{\lambda+1}{2}; 2; 1-x\right){\bf 1}_{[0,1]}(x).
\end{gather*}
\end{Proposition}

\begin{proof}
Let $\lambda >0$ and assume \eqref{Rel} holds for some probability measure $\mu = \mu_{\lambda}$, namely,
\begin{gather*}
\int \frac{1}{(1-wx)^{\lambda}}\mu(dx) = \frac{1}{(1-w^2)^{\lambda}} = \sum_{k \geq 0}\frac{(\lambda)_k}{k!} w^{2k}, \qquad w = 1/z \in (0,1).
\end{gather*}
Hence, $\mu$ is symmetric and its even moments are given by
\begin{gather*}
\int x^{2k} \mu(dx) = \frac{(2k)!(\lambda)_k}{k!(\lambda)_{2k}}.
\end{gather*}
Using Legendre duplication formula, these moments may be written as
\begin{gather*}
\frac{\Gamma(k+\lambda)\Gamma(k+1/2)}{2^{\lambda-1}\Gamma(k+(\lambda/2))\Gamma(k+(\lambda+1)/2)} = \frac{(\lambda)_k(1/2)_k}{(\lambda/2)_k((\lambda+1)/2)_k}.
\end{gather*}
According to~\cite[Theorem~6.2]{Duf1}, there exists a probability distribution $\chi = \chi_{\lambda}$ supported in~$(0,1)$ such that
\begin{gather*}
\int_0^1x^k \chi(dx) = \frac{(\lambda)_k(1/2)_k}{(\lambda/2)_k((\lambda+1)/2)_k}
\end{gather*}
if and only if $\lambda \geq 1$. This is an instance of the so-called $G$-distributions~\cite{Duf,Duf1} and its Lebesgue decomposition is given by
\begin{gather*}
\chi(dx)
 = \frac{\Gamma((\lambda+1)/2)\Gamma(\lambda/2)}{\Gamma(1/2)\Gamma(\lambda)} \\
 \hphantom{\chi(dx)=}{} \times \left[\delta_1(dx) + \frac{\lambda(\lambda-1)}{4}x^{\lambda-1}\, {}_2F_1\left(\frac{\lambda}{2}+1, \frac{\lambda+1}{2}; 2; 1-x\right){\bf 1}_{[0,1]}(x)dx\right]
\\
\hphantom{\chi(dx)}{} = \frac{1}{2^{\lambda-1}} \left[\delta_1(dx) + \frac{\lambda(\lambda-1)}{4}x^{\lambda-1}\, {}_2F_1\left(\frac{\lambda}{2}+1, \frac{\lambda+1}{2}; 2; 1-x\right){\bf 1}_{[0,1]}(x)dx\right].
\end{gather*}
Noting $\xi$ is the push forward of $\mu$ under the square map $x \mapsto x^2$, we are done.
\end{proof}

\begin{Remark}
Some parts of Theorems 1 and~2 in~\cite{Duf} as well as Theorem~6.2 in~\cite{Duf1} appear also in~\cite{Kar-Pri}.
\end{Remark}

\subsection{The Shrinkage process}

While the Wigner distribution is the limiting transition distribution of Young diagrams drawn from the Plancherel measure, the arcsine distribution describes in this case the limiting shape of these diagrams (it is referred to as the Rayleigh measure~\cite{Ker1}). In turn, the latter is the transition distribution whose Rayleigh measure is the symmetric Bernoulli distribution and belongs to a more general family of distributions related with triangular diagrams and with the Shrinkage process~\cite{Ker}. The analogue of example~\eqref{Examp2} is derived as follows. Consider~\eqref{GST1} with $\lambda = \gamma+\beta + 2$:
\begin{gather*}
\frac{1}{(z-1)^{\lambda}}\, {}_2F_1\left(\lambda, \gamma+1, \lambda, \frac{2}{1-z}\right) = \frac{1}{(z-1)^{\lambda}} \left(\frac{1-z}{-z-1}\right)^{\gamma+1}= \frac{1}{(z-1)^{\lambda-\gamma-1}}\frac{1}{(z+1)^{\gamma+1}}
\end{gather*}
for $z > 1$. Write $\gamma + 1 = \lambda p$ for some $0 < p < 1$, then
\begin{gather*}
\int \frac{1}{(z-x)^{\lambda}} \mu_{p\lambda- 1, (1-p)\lambda -1}(dx) = \left[\frac{1}{(z-1)^{1-p}(z+1)^{p}}\right]^{\lambda} = \left[\int \frac{1}{z-x} \mu_{p -1, -p}(dx)\right]^{\lambda}.
\end{gather*}
In particular, if $p=1/2$ then
\begin{gather*}
\frac{\Gamma(\lambda)}{2^{\lambda-1}[\Gamma(\lambda/2)]^2} \int_{-1}^1 \frac{1}{(z-x)^{\lambda}} \big(1-x^2\big)^{(\lambda/2)-1}dx = \left[\int_{-1}^1 \frac{1}{z-x} \frac{dx}{\pi\sqrt{1-x^2}}\right]^{\lambda}.
\end{gather*}

\subsection{The free Poisson distribution}

Let $\mu$ be a (non necessarily symmetric) probability distribution supported in~$[-1,1]$ and denote~$\kappa_{\lambda}$ the probability distribution with beta density
\begin{gather*}
\frac{\Gamma(\lambda)}{[\Gamma(\lambda/2)]^2} [x(1-x)]^{(\lambda/2)-1}{\bf 1}_{[0,1]}(x).
\end{gather*}
Denote $\kappa_{\lambda} \star \mu$ the multiplicative convolution of~$\kappa_{\lambda}$ and~$\mu$, that is, the probability distribution of the product of two independent random variables with probability distributions~$\kappa_{\lambda}$ and~$\mu$. By def\/inition,
\begin{gather*}
\int f(x) \left(\kappa_{\lambda} \star \mu\right)(dx) = \iint f(uv)\kappa_{\lambda}(du) \mu(dv)
\end{gather*}
for any bounded measurable function~$f$. Then, for $|z| > 1$, the integral representation~\eqref{Euler} yields
\begin{align*}
\int \frac{1}{(z-x)^{\lambda}} \kappa_{\lambda} \star \mu (dx) &= \frac{\Gamma(\lambda)}{[\Gamma(\lambda/2)]^2 z^{\lambda}} \int \left(\int_0^1 \frac{1}{(1-(v/z)u)^{\lambda}} [u(1-u)]^{(\lambda/2)-1}du \right)\mu(dv)
\\ &= \frac{1}{z^{\lambda}} \int {}_2F_1\left(\lambda, \frac{\lambda}{2}, \lambda; \frac{v}{z} \right) \mu(dv)
= \frac{1}{z^{\lambda/2}}\int \frac{1}{(z-v)^{\lambda/2}}\mu(dv).
\end{align*}
In particular, if $\mu = \mu_{\lambda}$ satisf\/ies \eqref{Rel} with the exponent $\lambda/2$
\begin{gather*}
\int \frac{1}{(z-x)^{\lambda/2}} \mu_{\lambda/2}(dx) = \left[G_{\nu}(z)\right]^{\lambda/2}
\end{gather*}
for some probability measure $\nu$ independent of $\lambda$, then
\begin{gather*}
\int \frac{1}{(z-x)^{\lambda}} \kappa_{\lambda} \star \mu (dx) = \left[\frac{[G_{\nu}(z)]^{1/2}}{z^{1/2}}\right]^{\lambda}
\end{gather*}
at least for $z > 1$. Since $\Im\big([G_{\nu}(z)]^{1/2}/z^{1/2}\big) < 0$ for $\Im(z) > 0$, then the last equality extends analytically to the upper half-plane. Moreover,
\begin{gather*}
\lim_{y \rightarrow \infty} (iy) \frac{[G_{\nu}(iy)]^{1/2}}{(iy)^{1/2}}= 1
 \end{gather*}
so that $z \mapsto [G_{\nu}(z)]^{1/2}/z^{1/2}$ is a Nevannlina--Pick function and may be represented as the Stieltjes transform of a probability measure (see for instance~\cite[Lemma~2.2]{Sho-Tam}). This elementary observation allows to derive further examples of probability distributions $(\mu, \nu)$ satisfying~\eqref{Rel}. For instance,

\begin{Proposition}
For any $\lambda > 0$ and $z$ in the upper half-plane,
\begin{align*}
\int \frac{1}{(z-x)^{\lambda}} \kappa_{\lambda} \star \kappa_{\lambda+1} (dx) = \frac{2^{\lambda}}{[z + \sqrt{z(z-1)}]^{\lambda}}.
\end{align*}
\end{Proposition}

\begin{proof}
Let $\mu = \kappa_{\lambda+1}$, then for $z$ in a suitable region in the upper half-plane,
\begin{align*}
\int \frac{1}{(z-x)^{\lambda}} \kappa_{\lambda} \star \kappa_{\lambda+1} (dx) &= \frac{\Gamma(\lambda+1)}{[\Gamma(\lambda+1/2)]^2} \frac{1}{z^{\lambda/2}}\int \frac{1}{(z-x)^{\lambda/2}}[x(1-x)]^{(\lambda+1)/2-1}dx
\\ & = \frac{1}{z^{\lambda}}\, {}_2F_1\left(\frac{\lambda}{2}, \frac{\lambda+1}{2}, \lambda+1; \frac{1}{z} \right)
\\ & = \frac{2^{\lambda}}{[\sqrt{z}(\sqrt{z} + \sqrt{z-1})]^{\lambda}} = \frac{2^{\lambda}}{[z + \sqrt{z(z-1)}]^{\lambda}},
\end{align*}
where the third equality follows from \eqref{Id1}. But, the map
\begin{gather*}
z \mapsto \frac{2}{z + \sqrt{z(z-1)}} = 2\frac{z - \sqrt{(z-(1/2))^2- (1/4)}}{z}
\end{gather*}
is the Stieltjes transform of the free Poisson distribution with parameters $(1, 1/4)$ (see for instance \cite[p.~204]{NS}) whose density reads
\begin{gather*}
\frac{2}{\pi} \sqrt{\frac{1-x}{x}}{\bf 1}_{[0,1]}(x).
\end{gather*}
Therefore, the proposition follows by analytic continuation.
\end{proof}

\begin{Remark}
The free Poisson distribution appears the limiting transition measure of random diagrams arising from the decomposition of tensor product representations of the symmetric group~\cite{Biane1}. On the other hand, the density of $\kappa_{\lambda} \star \kappa_{\lambda+1}$ is given by~\cite{Duf}
\begin{gather*}
\frac{\Gamma(\lambda)\Gamma(\lambda+1)}{\Gamma(\lambda/2)\Gamma((\lambda+1)/2)\Gamma(\lambda +(1/2))}x^{(\lambda/2)-1}(1-x)^{\lambda - 1/2}\\
\qquad{}\times {}_2F_1\left(\frac{\lambda-1}{2}, \frac{\lambda+1}{2}, \lambda+\frac{1}{2}, 1-x\right) {\bf 1}_{[0,1]}(x).
\end{gather*}
\end{Remark}

\section{Further developments}
The Humbert polynomials $\big(H_{n}^{(\alpha,d)}\big)_{n \geq 0}$ of parameters $\alpha > -1/2$, $\alpha \neq 0$, $d \in \mathbb{N} {\setminus} \{0\}$ and degrees $n \in \mathbb{N}$, are def\/ined by their generating series:
\begin{gather*}
\sum_{n \geq 0} H_{n}^{(\alpha,d)} z^n = \frac{1}{(1-(d+1)xz + z^{d+1})^{\alpha}},
\end{gather*}
and reduce to ultraspherical polynomials $\big(C_n^{(\alpha)}\big)_n$ when $d=1$. They cease to be orthogonal as soon as $d \geq 2$ and are rather $d$-orthogonal in the following sense (see, e.g.,~\cite{Lam-Oun}): there exist $d$ linear functionals $\Gamma_0^{(\alpha,d)}, \dots, \Gamma_{d-1}^{(\alpha,d)}$ on the space of polynomials such that
\begin{gather*}
\Gamma_k\big(H_n^{(\alpha,d)}H_r^{(\alpha,d)}\big) = 0, \qquad r > dn +k, \qquad
\Gamma_k\big(H_n^{(\alpha,d)}H_{nd+k}^{(\alpha,d)}\big) \neq 0,
\end{gather*}
for all $n \in \mathbb{N}$, $0 \leq k \leq d-1$. In particular, $\Gamma_0^{(\alpha,d)}\big(H_0^{(\alpha,d)}H_r^{(\alpha,d)}\big) = 0$ for any $r > 0$ so that
\begin{gather*}
\int \frac{1}{(1-(d+1)xz + z^{d+1})^{\alpha}}\Gamma_0(dx) = 1.
\end{gather*}
Hence, we can f\/ind a suitable region outside the compact support of $\Gamma_0^{(\alpha,d)}$ for which
\begin{gather*}
\int \frac{1}{(f_{\alpha}(z)- x)^{\alpha}}\Gamma_0(dx) = [(d+1)z]^{\alpha},
\qquad \text{where} \qquad f_{\alpha}(z):= \frac{1+z^{d+1}}{(d+1)z}.
\end{gather*}
Now, the trinomial equation $f_{\alpha}(z) = y$ has $(d+1)$ roots which may be expressed through the Gauss hypergeometric function~\cite{Gla,Hil} and we seek the one which tends to zero when $y$ tends to inf\/inity in the upper half-plane. For instance, when $d=2$ then the polynomial equation
\begin{gather*}
z^3- 3yz + 1 = 0, \qquad \Im(y) > 0,
\end{gather*}
may be transformed by setting $z = \sqrt{-y} w$ into
\begin{gather*}
w^3 + 3w + \frac{1}{(-y)^{3/2}} = 0.
\end{gather*}
From \cite[p.~265]{Hil} the sought root is given by
\begin{gather*}
z = -\frac{(-y)^{1/2}}{3(-y)^{3/2}} \, {}_2F_1\left(\frac{1}{3}, \frac{2}{3}, \frac{3}{2}; -\frac{1}{4((-y)^{3/2})^2}\right) = \frac{1}{3y} \, {}_2F_1\left(\frac{1}{3}, \frac{2}{3}, \frac{3}{2}; \frac{1}{4y^3}\right),
\end{gather*}
where the last equality is valid for $y$ lying in some sector in the upper half-plane. Using the Euler integral representation
\begin{gather*}
{}_2F_1\left(\frac{1}{3}, \frac{2}{3}, \frac{3}{2}; \frac{1}{4y^3}\right) = \frac{4^{1/3}\Gamma(3/2)}{\Gamma(2/3)\Gamma(5/6)} \int_0^1 \frac{y}{(y^3-x)^{1/3}} x^{-1/3}(1-x)^{-1/6}dx,
\end{gather*}
we get the following equality
\begin{gather*}
\int \frac{1}{(y- x)^{\alpha}}\Gamma_0^{(\alpha,2)}(dx) = \left\{\frac{4^{1/3}\Gamma(3/2)}{\Gamma(2/3)\Gamma(5/6)} \int_0^1 \frac{1}{(y^3-x)^{1/3}} x^{-1/3}(1-x)^{-1/6}dx\right\}^{\alpha}.
\end{gather*}
More generally, the solution to the equation
\begin{gather*}
z^{d+1} + (d+1)z + \frac{1}{(-y)^{(d+1)/d}} = 0, \qquad \Im(y) > 0,
\end{gather*}
tending to zero as $y \rightarrow \infty$ may be derived along the same lines written in~\cite[p.~266]{Hil}, and may be expressed through the hypergeometric series
\begin{gather*}
{}_dF_{d-1} \left(\frac{i}{d+1}, \, 1 \leq i \leq d, \, \frac{i+1}{d}, \, 1 \leq i \leq d, \, i \neq d-1; \, \frac{(-1)^d}{d^dy^{d+1}}\right).
\end{gather*}
Thus, the reasoning above applies and leads to the generalized Stieltjes transform of $\Gamma_0^{(\alpha,d)}$ as a~$\alpha$-power of this hypergeometric series.
\begin{Remark}
An anonymous referee pointed out to the author that the functional $\Gamma_0^{(\alpha,d)}$ coincides up to the variable change $x \mapsto d^dx^{d+1}$ with the representative measure of the hyper\-geo\=metric function ${}_{d+1}F_{d}$ as a generalized Stieltjes transform~\cite[Theorem~2]{Kar-Pri1}. Actually, $\Gamma_0^{(\alpha,d)}$ is expressed in terms of the Meijer $G$-function~\cite[Theorem 2.4]{Lam-Oun}
\begin{gather*}
\frac{1}{x}G_{d+1, d+1}^{d+1,0}\left(d^dx^{d+1} \Big| \begin{array}{@{}c@{}}(\alpha+1)/d,\dots, (\alpha+d)/d, 1 \\ 1/(d+1), \dots, d/(d+1), 1\end{array} \right)
\end{gather*}
which reduces to
\begin{gather*}
\frac{1}{x}G_{d, d}^{d,0}\left(d^dx^{d+1} \Big| \begin{array}{@{}c@{}} (\alpha+1)/d,\dots, (\alpha+d)/d \\ \displaystyle 1/(d+1), \dots, d/(d+1)\end{array}\right),
\end{gather*}
while Theorem 2 in \cite{Kar-Pri1} entails
\begin{gather*}
{}_{d+1}F_{d}\left(\alpha, a_1,\dots, a_d, b_1, \dots, b_d, \frac{1}{y}\right) = \prod_{i=1}^d \frac{\Gamma(b_i)}{\Gamma(a_i)} \int_0^1 \frac{y^{\alpha}}{(y-x)^{\alpha}}G_{d,d}^{d,0}\left(x \Big| \begin{array}{@{}c@{}} b_1, \dots, b_d \\ \displaystyle a_1, \dots, a_d \end{array}\right) \frac{dx}{x}.
\end{gather*}
for $y \in \mathbb{C} {\setminus} (-\infty, 1]$.
\end{Remark}

\subsection*{Acknowledgements}

The author would like to thank H.~Cohl and C.~Dunkl for their helpful remarks on earlier versions of the paper. He also greatly appreciates the suggestions of the anonymous referees which considerably improved the presentation.

\pdfbookmark[1]{References}{ref}
\LastPageEnding

\end{document}